\documentclass[a4paper]{amsart}
\usepackage{amssymb}
\usepackage{amsmath}
\usepackage{mathrsfs}
\usepackage{amsthm}
\usepackage[numbers]{natbib}
\usepackage{dsfont}
\usepackage{enumitem}
\usepackage{expdlist}
\usepackage{mathtools}
\theoremstyle{definition} \newtheorem{lemma}{Lemma}
\theoremstyle{definition} \newtheorem{theorem}[lemma]{Theorem}
\theoremstyle{remark} \newtheorem*{notation}{Notation}
\theoremstyle{remark} \newtheorem*{acknowledgements}{Acknowledgements}
\title{Singular Sets and the Lavrentiev Phenomenon}
\author{Richard Gratwick}
\address{Mathematics Institute, Zeeman Building, University of Warwick, Coventry, CV4 7AL, UK}
\email{r.gratwick@warwick.ac.uk}
\thanks{This research was conducted during PhD study funded by the EPSRC, and this paper was prepared with support from the European Research  Council under the European Union's Seventh Framework Programme (FP/2007-2013) / ERC Grant Agreement n.291497.}
\date{\today}
\begin{document}
\begin{abstract}
We show that non-occurrence of the Lavrentiev phenomenon does not imply that the singular set is small.  Precisely, given a compact Lebesgue null subset $E \subseteq \mathbb{R}$ and an arbitrary superlinearity, there exists a smooth, strictly convex Lagrangian with this superlinear growth, such that all minimizers of the associated variational problem have singular set exactly $E$, but still admit approximation in energy by smooth functions.
\end{abstract}
\maketitle
\section{Introduction}
For a fixed closed bounded interval $[a,b] \subseteq \mathbb{R}$, we consider the problem of minimizing the functional
\begin{equation} 
\label{problem}
\mathscr{L}(u) = \int_a^b L(x, u(x), u'(x))\, dx
\end{equation}
over the class of real-valued absolutely continuous functions $u \in \mathrm{AC}(a,b)$ with fixed boundary conditions, where the function $L = L(x,y,p) \colon \mathbb{R}^3 \to \mathbb{R}$, the {\it Lagrangian}, is a fixed function of class $C^{\infty}$.    The first general existence results were given by~\citet{Tonelli-1915,Tonelli-Fondamenti-2}; these require the assumptions of superlinearity and convexity of $L$ in the variable $p$.   Assuming the stronger condition that $L_{pp} > 0$, he also proved the following {\it partial regularity} theorem: minimizers of~\eqref{problem} are everywhere differentiable (possibly with infinite derivative) and this derivative is continuous as a map into the extended real line.  Thus the {\it singular set} of a minimizer, defined as those points where the derivative is infinite, is closed. Since the minimizer is absolutely continuous, we know immediately that it must also be of Lebesgue measure zero.  A number of versions of Tonelli's partial regularity theorem, under significantly weaker hypotheses than Tonelli's original statement, can be found in the work of~\citet{Clarke-Vinter-1985-regularity,Clarke-Vinter-1985-small,Sychev-1993,Csornyei-etal-2008}, and~\citet{Ferriero-2012,Ferriero-2013}. \citet{Gratwick-Preiss-2011} show that little further improvement is possible.

Tonelli proved conditions guaranteeing that the singular set is empty, i.e.\ that the minimizer is fully regular.  That it can be non-empty given the assumption necessary for partial regularity (i.e.\ $L_{pp} > 0$) was shown by~\citet{Ball-Mizel-1985}, who exhibited examples of minimizers with one-point singular sets.  They also constructed, given an arbitrary closed set of measure zero $E$, a $C^{\infty}$ Lagrangian depending only on $(y,p)$, superlinear in $p$ and with $L_{pp} > 0$, such that the unique minimizer of~\eqref{problem} has singular set precisely $E$.

\citet{Davie-1988} completed this work by constructing, for an arbitrary closed null set $E$, a $C^{\infty}$ Lagrangian $L = L(x,y,p)$, superlinear in $p$ and with $L_{pp} > 0$, such that any minimizer has singular set exactly $E$.  Davie constructs an admissible function $v \in \mathrm{AC}(a,b)$ and a Lagrangian $L$ so that there exists a constant (in his notation) $(8\alpha)^{-1}> 0 $ such that
$
\mathscr{L}(v) < (8 \alpha)^{-1}
$,
but for any admissible function $u \in \mathrm{AC}(a,b)$, if for some $c \in E$ we have that $u'(c)$ exists and is finite, then
$
\mathscr{L}(u) \geq (8 \alpha)^{-1}
$.
Therefore any minimizer (and at least one exists) must have infinite derivative on the set $E$.  Thus the proof rests on the fact that the energy of $C^1$ functions is bounded away from the infimum of the energy over all $\mathrm{AC}(a,b)$ functions, i.e.\ that the {\it Lavrentiev phenomenon} occurs.  That such a gap can occur at all was first shown by~\citet{Lavrentiev-1926}.  Since the corresponding example of~\citeauthor{Ball-Mizel-1985} described above is autonomous, i.e. has no dependence on the variable $x$, it follows by a result of~\citet{Alberti-SerraCassano-1994} that there can be no Lavrentiev gap in this example.

This raises the question of the exact relationship between the singular set and the occurrence of the Lavrentiev phenomenon.  If a  problem exhibits the Lavrentiev phenomenon, then certainly the singular set of any minimizer over $\mathrm{AC}(a,b)$ must be non-empty, although it should be noted that the first examples of such problems found by~\citet{Lavrentiev-1926} and~\citet{Mania-1934} do not satisfy the $L_{pp} > 0$ condition required for classical partial regularity statements.  That a minimizer has a non-empty singular set does not, of course, in general imply the occurrence of a Lavrentiev gap.  Quite the reverse is in fact the case: one usually has to go to some effort to prove that a Lavrentiev gap does occur.  However, it might be conjectured that if a minimizer has a {\it large} singular set, for example of Hausdorff dimension one, then a gap must occur.  Thus the question is: can one prove Davie's result without inducing a Lavrentiev gap?  We show, using the methods which~\citet{Csornyei-etal-2008} introduced in the context of  universal singular sets, that this is indeed possible, i.e.\ that the existence of a large singular set does {\it not} imply occurrence of the Lavrentiev phenomenon.  Conversely, knowing that the Lavrentiev phenomenon does not occur does {\it not} tell us that the minimizer has small singular set, for example in the sense of Hausdorff dimension, nor indeed give us any information about the nature of the singular set not already available.

The methods of~\citeauthor{Csornyei-etal-2008} also naturally allow us to construct a Lagrangian giving this result which has arbitrary given superlinear growth, so this result is a generalization of Davie's result even without the further result preventing a Lavrentiev gap.

We prove the following theorem.
\begin{theorem}
\label{sing:main}
Let $[a,b]$ be a closed bounded subinterval of the real line, and let $E \subseteq [a,b]$ be closed and Lebesgue null.   Let $\omega \in C^{\infty}(\mathbb{R})$ be  strictly convex, such that $\omega(p) \geq \omega(0) = 0$ for all $p \in \mathbb{R}$, and $ \omega(p) / |p| \to \infty$ as $|p| \to \infty$ (i.e.\ $\omega$ has superlinear growth).

Then there exists $L \in C^{\infty}(\mathbb{R}^3)$, $L = L(x,y,p)$, strictly convex in $p$ and such that $L(x,y,p) \geq \omega(p)$ for all $(x,y,p) \in \mathbb{R}^3$, and function $u \in \mathrm{AC}(a,b)$ such that
\begin{itemize}
\item $u$ is the unique minimizer of the functional~\eqref{problem} with respect to its own boundary conditions;
\item the singular set of $u$ is precisely $E$; and
\item there exist admissible functions $u_k \in C^{\infty}([a,b])$ (i.e.\ $u_k (a) = u(a)$ and $u_k (b) = u(b)$)
such that $u_k \to u$ uniformly and $\mathscr{L}(u_k) \to \mathscr{L}(u)$.
\end{itemize}
\end{theorem}

For the entire paper we shall assume that $[a,b]$, $\emptyset \neq E \subseteq [a,b]$, and $\omega$ are fixed as in Theorem~\ref{sing:main}.
\begin{notation}
We let $\| \cdot \|$ denote the supremum norm on $\mathbb{R}^2$, which is the norm used throughout and for the following definitions.  The diameter $\mathrm{diam}(X) \in [0, \infty)$ of a bounded set $X \subseteq \mathbb{R}^2$ is defined by $\mathrm{diam} (X) = \sup_{x, y \in X} \| x- y\|$.  For sets $X,Y \subseteq \mathbb{R}^2$, the notation $X \Subset Y$ is used when the closure $\overline{X}$ of $X$ is compact and contained in $Y$, and the distance $\mathrm{dist}(X,Y) \in [0, \infty]$ between the two sets is defined by $\mathrm{dist}(X,Y) = \inf_{x \in X,\ y \in Y} \| x - y\|$, and is written $\mathrm{dist} (x, Y)$ when $X = \{ x\}$ (this is understood to be $+ \infty$ if one of the sets is empty).  On the real line, for $r>0$, we will use $B_r (X)$ for the $r$-neighbourhood of a subset $X \subseteq \mathbb{R}$.

For a bounded interval $[a,b]$ in $\mathbb{R}$, we shall write $\mathrm{AC}(a,b)$ for the class of absolutely continuous functions on $[a,b]$.  For any function $u \colon \mathbb{R} \to \mathbb{R}$ we let $U\colon \mathbb{R} \to \mathbb{R}^2$ be given by $U(x) = (x, u(x))$.  The supremum norm of a function on $\mathbb{R}^2$ shall be denoted by $\| \cdot \|_{\infty}$. Partial derivatives shall be denoted by subscripts, e.g.\ $\Phi_x$, $\Phi_y$  for functions $\Phi = \Phi (x,y) \colon \mathbb{R}^2 \to \mathbb{R}$.  The Lebesgue measure on the real line shall be denoted by $\lambda$.
 \end{notation}
 \begin{acknowledgements}
 I wish to thank David Preiss for his insightful advice on this subject and this paper.  
 \end{acknowledgements}
 \section{Calibration}
Our approach to the construction of minimizers with infinite derivatives is inspired by that in~\citet{Csornyei-etal-2008}.  We use a calibration argument to prove that functions with a specified derivative are minimizers of~\eqref{problem} where the Lagrangian $L$ is constructed via a potential defined on $\mathbb{R}^2$.  The original context of this method was the study of universal singular sets, specifically the construction of a Lagrangian with universal singular set containing a certain subset $S$ of the plane.  Thus~\citeauthor{Csornyei-etal-2008}\ constructed the potential  to have singular behaviour at these points $S$.  For each point in $S$ a minimizer was constructed with derivative given via the potential (hence infinite at that point) and graph passing through that point.

We need just one minimizer $u$, but one that has infinite derivative at every point of the set $E$.  Thus it is more natural to begin by defining $u$ (via its derivative), because firstly this is very easy, and secondly this readily gives us a sequence of smooth admissible functions approximating $u$ with which we shall see the Lavrentiev phenomenon does not occur.  So we approach the construction of the Lagrangian with the derivative of our intended minimizer already given, and with this derivative construct a function $\psi$ on the plane with which we can compare the potential.  This is then the reverse logic to that used in~\citeauthor{Csornyei-etal-2008}, in which minimizers were selected to solve an ODE given via the potential.  Our function $\psi$ is defined to mimic this idea in the sense that it agrees with the derivative of $u$ on the graph of $u'$; so our minimizer does satisfy (almost everywhere) the ODE $u' = \psi ( x, u)$.  This is however a consequence of our definition of $\psi$ given $u$, not vice versa.

We first recall Lemma~10 from~\citet{Csornyei-etal-2008}, stated and used almost as in this original paper, except that later we need also an upper bound of the function, for our smooth approximation estimates. We do not repeat the (simple) proof of the other statements.
\begin{lemma}
\label{sing:cornersmoothing}
There exists a  $C^{\infty}$ function $\gamma \colon \{ (p,a,b) \in \mathbb{R}^3 : b >0 \} \to \mathbb{R}$ with the following properties:
\begin{enumerate}[label=(\thelemma.\arabic*), leftmargin=*]
\item $ p \mapsto \gamma (p,a,b)$ is convex; \label{sing:4.4.a}
\item $\gamma (p,a,b) = 0$ for $p \leq a-1$; \label{sing:4.4.b}
\item $\gamma(p,a,b) = b(p-a)$ for $p \geq a+1$; \label{sing:4.4.c}
\item $\gamma (p,a,b) \geq \max\{0, b(p-a)\}$; and \label{sing:4.4.d}
\item $ \gamma (p,a,b) \leq b |p-a + 1|$. \label{sing:4.4.e}
\end{enumerate}
\end{lemma}
\begin{proof}
Recalling the proof from~\citet{Csornyei-etal-2008}, we see
$
\gamma(p,a,b) = b\int_{\infty}^{p-a} \eta
$,
where non-decreasing $\eta \in C^{\infty}(\mathbb{R})$ was chosen such that $\eta (x) = 0 $ if $ x \leq -1$, $\eta(x) = 1$ if $x \geq 1$, and $\int_{-1}^1 \eta = 1$.
The only new statement~\ref{sing:4.4.e} is trivial: if $p \leq a-1$ or $p \geq a+1$ then the result follows from~\ref{sing:4.4.b} or~\ref{sing:4.4.c} respectively.  If $a-1 \leq p \leq a+1$, then
\begin{equation*}
\gamma (p,a,b)  = b\! \int_{\infty}^{p-a} \!\eta (x)\, dx
 \leq b \!\int_{-1}^{p-a}\! 1 \, dx
= b(p-a +1)
\leq b|p-a +1|.\qedhere
\end{equation*}
\end{proof}
The next result is a version of Lemma~11 in~\citet{Csornyei-etal-2008}.  The main difference, as discussed, is that $\psi$ is given before the potential $\Phi$.  We recall that for a function $u \colon [a,b] \to \mathbb{R}$, the function $U \colon [a,b] \to \mathbb{R}^2$ is given by $U(x) = (x, u(x))$.
\begin{lemma}\label{lemma11'}
Let $S \subseteq \mathbb{R}^2$ be compact, $\psi \in C^{\infty}(\mathbb{R}^2  \setminus S)$ be such that  $\psi (x,y) \to \infty$ as $\mathrm{dist}((x,y), S) \to 0$, and $\Phi \in C^{\infty}(\mathbb{R}^2 \setminus S) \cap C(\mathbb{R}^2)$ satisfy the following conditions:
\begin{enumerate}[label=(\thelemma.\arabic*), leftmargin=*]
\item $- \Phi_x (x,y) \geq 4 \Phi_y (x,y)> 0$ for all $(x,y) \in \mathbb{R}^2 \setminus S$; \label{11beta}
\item $ \Phi_y (x,y) > 320\omega'(\psi(x,y))$  for all $(x,y) \in \mathbb{R}^2 \setminus S$; \label{11gamma}
\item $ -2 (\Phi_x / \Phi_y) (x,y) \leq \psi(x,y) \leq -160 (\Phi_x / \Phi_y) (x,y)$  for all $(x,y) \in \mathbb{R}^2 \setminus S$;  and \label{11delta}
\item for all $u \in \mathrm{AC}(a,b)$, the sets $U^{-1}(S)$ and $(\Phi \circ U)(U^{-1}(S))$ are Lebesgue null.\label{11epsilon}
\end{enumerate}
Then there exists a Lagrangian $L \in C^{\infty}(\mathbb{R}^3)$, strictly convex in $p$ and satisfying $L(x,y,p) \geq \omega(p)$ for all $(x,y,p) \in \mathbb{R}^3$, such that for all $u \in \mathrm{AC}(a,b)$,
$$
\mathscr{L}(u) = \int_a^b L(x, u(x), u'(x))\, dx \geq \Phi(U(b)) - \Phi (U(a)),
 $$
 with equality if and only if $u'(x) = \psi(x, u(x))$ for almost every $x \in [a,b]$.  In particular, any such $u$ is the unique minimizer of~\eqref{problem} with respect to its boundary conditions.
 \end{lemma}
\begin{proof}
This mimics the proof of Lemma~11 in~\citet{Csornyei-etal-2008}.
Define $\theta, \xi \in C^{\infty}(\mathbb{R}^2 \setminus S)$ by
$$
\theta = \Phi_y - \omega'(\psi)\ \textrm{and}\ \xi = \frac{-\Phi_x + \omega(\psi) - \omega'(\psi) \psi}{\theta }.
$$

Fix $(x,y) \in \mathbb{R}^2 \setminus S$.  Then note by~\ref{11delta} and~\ref{11beta} that $\psi > 0$, so by properties of $\omega$ we have that $\omega'(\psi) > 0$.  So using also~\ref{11gamma} we have that
\begin{equation} \label{thetaphiy}
\theta > \Phi_y - \frac{1}{320}\Phi_y = \frac{319}{320}\Phi_y > 0,
\end{equation}
 so $\xi$ is well-defined. By convexity of $\omega$ we have that $\omega(p) - \omega'(p) p \leq \omega(0) = 0$ for all $p \geq 0$.  So, using this and properties~\ref{11delta} and~\ref{11gamma}, we see
\begin{equation}
 -\Phi_x  \geq -\Phi_x + \omega(\psi) - \omega'(\psi)\psi  = \xi \theta \geq - \Phi_x - \omega'(\psi)\psi
\geq  -\Phi_x + \frac{\Phi_y}{320}\cdot \frac{160 \Phi_x}{\Phi_y} = - \Phi_x / 2. \label{xxix}
\end{equation}
So, since $\Phi_y = \theta + \omega'(\psi) > \theta > 0$, we see by~\ref{11beta} that
\begin{equation}
\xi \geq -\Phi_x / (2 \theta) \geq - \Phi_x / (2 \Phi_y) \geq 2
, \label{xibignearS}
\end{equation}
and so, using~\ref{11delta},~\eqref{thetaphiy}, and~\eqref{xxix},
$$
\psi  \geq -2\Phi_x/\Phi_y \geq -2 \cdot 319 \Phi_x/ (320 \theta) \geq 3 \xi / 2 \geq \xi + 1.
$$
The point of these estimates, and the choice of constants in the assumptions which allows them to be derived, is that
\begin{gather}
0 \leq \xi -1 \label{xi>1} \shortintertext{and}
 \psi  \geq \xi + 1. \label{xi<psi-1}
\end{gather}
We use the corner-smoothing function $\gamma$ from Lemma~\ref{sing:cornersmoothing} to define
$$
F(x,y,p) =
\begin{cases}
 \gamma (p, \xi (x,y), \theta (x,y)) & (x,y) \in \mathbb{R}^2 \setminus S, \\
0 & \textrm{otherwise}.
\end{cases}
$$
Clearly $F \in C^{\infty}((\mathbb{R}^2 \setminus S) \times \mathbb{R})$.  For fixed $p \in \mathbb{R}$, by the growth assumption on $\psi$ there exists an open set $\Omega \supseteq S$ such that $\psi > 320(p+2)$ on $\Omega$.  By~\eqref{xibignearS} and~\ref{11delta} we see that $ \xi \geq - \Phi_x /(2 \Phi_y ) \geq \psi /320 \geq p + 2$, and so $F=0$ on $\Omega \times (-\infty, p + 1)$, by property~\ref{sing:4.4.b} of $\gamma$.  So in fact $F \in C^{\infty}(\mathbb{R}^3)$.  Clearly $F \geq 0$ by~\ref{sing:4.4.d}, and is convex in $p$ by~\ref{sing:4.4.a}.

Defining $L(x,y,p) = F(x,y,p) + \omega(p)$ gives a Lagrangian $L \in C^{\infty}(\mathbb{R}^3)$ such that $L \geq \omega$ and $L$ is strictly convex in $p$. For $(x,y) \in \mathbb{R}^2 \setminus S$, we have, by convexity of $\omega$ and property~\ref{sing:4.4.d} of $\gamma$, that
\begin{align*}
L(x,y,p) & \geq \omega(\psi(x,y)) + \omega'(\psi(x,y)) (p- \psi (x,y)) + \theta (x,y) ( p - \xi (x,y)) \\
& = \Phi_x (x,y) + p \Phi_y (x,y).
\end{align*}
Moreover, $p = \psi(x,y)$ implies equality by~\eqref{xi<psi-1} and~\ref{sing:4.4.c}; and equality in this inequality implies $p = \psi(x,y)$  by strict convexity of $\omega$.  Thus equality holds in this inequality if and only if $p = \psi (x,y)$.

Let $u \in \mathrm{AC}(a,b)$. Since $\Phi \in C^{\infty}(\mathbb{R}^2 \setminus S)$, we see that $(\Phi \circ U) \colon [a,b] \to \mathbb{R}$ is differentiable for almost every $x \notin U^{-1}(S)$, which is almost everywhere on $[a,b]$ by~\ref{11epsilon}, with derivative $(\Phi \circ U)'(x) = \Phi_x(U(x)) + u'(x)(\Phi_y (U(x))$, and for almost every $ x \in [a,b]$, the above inequality implies that
 \begin{equation}
  L(x, u(x), u'(x)) \geq  \Phi_x (x, u(x)) + u'(x) \Phi_y (x, u(x)) = (\Phi \circ U)'(x), \label{sing:key}
 \end{equation}
 with equality if and only if $u'(x) = \psi (x,u(x))$.   We note that $(\Phi \circ U)$ has the Lusin property, i.e.\ maps null sets to null sets:~\ref{11epsilon} implies that any subset of $U^{-1}(S)$ is mapped to a null set, and on $[a,b]\setminus U^{-1}(S)$ the function $(\Phi \circ U)$ is locally absolutely continuous, since $\Phi \in C^{\infty}(\mathbb{R}^2 \setminus S)$.

Let $\{(a_j, b_j)\}_{j \in J}$ be the (at most countable) sequence of components of $(a,b) \setminus U^{-1}(S)$ such that $\Phi (U(a_j)) < \Phi(U(b_j))$ (if there are no such components then the result is trivial).   Then using that $(\Phi \circ U)$ is locally absolutely continuous on $(a,b) \setminus U^{-1}(S)$ and the fact from~\ref{11epsilon} that $(\Phi \circ U)(U^{-1}(S))$ is null, we see, using~\eqref{sing:key}, that
 \begin{align*}
 \int_a^b L(x,u(x),u'(x)) \, dx &\geq \sum_{j \in J} \int_{a_j}^{b_j} L(x, u(x), u'(x))\, dx \\
 & \geq \sum_{j \in J} \int_{a_j}^{b_j} \max \{ 0, (\Phi \circ U)'\} \, dx \\
 & \geq \sum_{j \in J} \Phi ( U(b_j)) - \Phi ( U (a_j)) \\
 & \geq \Phi  ( U(b)) - \Phi ( U(a)).
 \end{align*}
 Equality in this relation implies that $L(x,u(x), u'(x)) = (\Phi \circ U)' (x) $ for almost every $x \in \bigcup_{j \in J}(a_j, b_j)$, but also that $\bigcup_{j \in J}(a_j, b_j) = (a,b) \setminus U^{-1}(S)$.  Therefore in fact $L(x,u(x), u'(x)) = (\Phi \circ U)' (x)$ for almost every $x \in (a,b) \setminus U^{-1}(S)$. By~\eqref{sing:key} this implies that $u'(x) = \psi(x,u(x))$ for almost every $x \in [a,b]$, since $U^{-1}(S)$ is null by~\ref{11epsilon}.

 Conversely, $u'(x) = \psi(x,u(x))$ almost everywhere implies  by~\ref{11delta} that
 $$
 (\Phi \circ U)'(x) = (\Phi_x \circ U)(x) + \psi(x) (\Phi_y \circ U)(x) \geq (-\Phi_x \circ U)(x) \geq 0
 $$
 almost everywhere.  This, combined with the fact that $(\Phi \circ U)$ has the Lusin property, implies that $(\Phi \circ U)$ is absolutely continuous~\citep[see][Chapter~IX, Theorem~7.7]{Saks-Theory}.  Moreover,~\eqref{sing:key} gives that $L(x,u(x), u'(x)) = (\Phi \circ U)'(x)$ almost everywhere, hence
 $$
 \int_a^b L(x, u(x), u'(x)) \, dx = \int_a^b (\Phi \circ U)'(x) \, dx = \Phi ( U(b)) - \Phi  (U(a)),
 $$
 as required.
\end{proof}
\section{Construction of the  minimizer}
We now begin the construction of our future minimizer $u$, by constructing first its derivative $\phi$.  The essential property of $\phi$ is that $\phi(x) \to \infty$ as $\mathrm{dist} (x, E) \to 0$.  We naturally define $\phi$ as the limit of a sequence of non-negative $C^{\infty}(\mathbb{R})$ functions $\{\phi_k\}_{k=0}^{\infty}$, where each $\phi_k$ is bounded above, and on an open set $V_k$ covering $E$ attains this bound (which tends to $\infty$ as $k \to \infty$).  We construct $\phi_k$ so that their primitives $u_k$ will be admissible functions in problem~\eqref{problem} (i.e.\ have the same boundary conditions as $u$) and converge uniformly to $u$.   In fact we shall guarantee that $u = u_k$ off $V_k$.  So, since our Lagrangian will be constructed as in Lemma~\ref{lemma11'}, our estimates  showing that there is no Lavrentiev gap reduce just to estimates of the integral over $V_k$ of a function involving the gradient of the potential $\Phi$.  This then requires a certain upper bound for the measure of $V_k$.  We must also remember that our potential $\Phi$ must have a gradient which satisfies inequalities involving $\phi$ and hence $\phi_k$.  This $\Phi$ will---just as in~\citet{Csornyei-etal-2008}---be defined using a sequence of $C^{\infty}(\mathbb{R}^2)$ functions $\{\Phi^k\}_{k=0}^{\infty}$ which have appropriately steep gradients on open sets $\Omega_k$ around $U(E)$.  To guarantee that these $\Phi^k$ converge, these sets must be small in the directions of these gradients, which is most easily achieved by ensuring they are small in all directions.  We choose $\Omega_k$ so that this measure is controlled by that of $V_k$; this gives another upper bound for the measure of $V_k$.  Other bounds are required for technical reasons in the proof; we impose just one inequality which suffices to give all the results.

For $k \geq 0$, let $\{h_k\}_{k=0}^{\infty}$ and $\{A_k\}_{k=0}^{\infty}$ be strictly increasing sequences of real numbers tending to infinity, such that $h_0, A_0 \geq 1$.  We will eventually need to define explicit values for these sequences to satisfy the exact inequalities required in Lemma~\ref{lemma11'}, but until we make these definitions, the construction requires only these general assumptions.

Define $V_0 = \mathbb{R}$ and $W_0 = (a-1,b+1)$.  For $k\geq 1$, we find  decreasing sequences of open sets $W_k \Subset V_k\subseteq \mathbb{R}$ covering $E$, of form
\begin{align*}
V_k  & =  \bigcup_{i=1}^{n_k}(a_k^i, b_k^i) \subseteq \mathbb{R},\ \text{where $a_k^1 < b_k^1 \leq a_k^2 < b_k^2 \leq \ldots a_k^{n_k} < b_k^{n_k}$, and}\\
W_k & = \bigcup_{i=1}^{n_k}(\tilde{a}_k^i, \tilde{b}_k^i) \subseteq \mathbb{R},\ \text{where $a_k^i < \tilde{a}_k^i < \tilde{b}_k^i < b_k^i $ for all $1 \leq i \leq n_k$},
\end{align*}
for some $n_k \geq 1$, such that
\begin{gather}
V_k \subseteq B_{2^{-k}}(E);\label{vkshrink}\\
V_k \Subset W_{k-1};\  \textrm{and}\label{vkcptdec}\\
\lambda(V_k) \leq \left((1 + 4(\mathrm{dist}(E, \mathbb{R} \setminus W_{k-1}))^{-1})2^{k+3}A_{k+1}^2(h_{k+1}+2)(1 + \omega (h_{k+1} + 2))n_{k-1}\right)^{-1}. \label{measvk}
\end{gather}
For $k \geq 1$ we define
\begin{itemize}
 \item $r_k^i = b_k^i - a_k^i > 0$, and $r_k = \min_{1 \leq i \leq n_k}r_k^i > 0$;
 \item $\tilde{r}_k^i = \tilde{b}_k^i - \tilde{a}_k^i > 0$, and $\tilde{r}_k = \min_{1 \leq i \leq n_k} \tilde{r}_k^i > 0$; and
 \item $\delta_k = \mathrm{dist} (E, \mathbb{R} \setminus V_k)$, where this is strictly positive by compactness of $E$.
 \end{itemize}
 Let $V_k^i = (a_k^i, b_k^i)$ and $W_k^i = (\tilde{a}_k^i, \tilde{b}_k^i)$ for each $1 \leq i \leq n_k$.  We assume that each component of $W_k$ contains a point of $E$.  Then since $\lambda (V_k) \to 0$ as $k \to \infty$, we see that $\bigcap_{k=1}^{\infty}V_k = E$.  We assume further that, as would be natural in the construction of the sets, that, when $E \subseteq (a,b)$, the sets $V_k \subseteq (a,b)$, and otherwise, i.e.\ when $a$ or $b \in E$, that the interval(s) covering the endpoint(s) are centred around the relevant endpoint(s), and that all the other intervals lie inside $(a,b)$.  Then in all cases,
 \begin{equation}
 \label{centred}
 \tilde{r}_k \leq 2 \lambda (W_k^j \cap (a,b)) 
 \end{equation} 
 for all $1 \leq j \leq n_k$ and all $k \geq 1$.
 
The sets $W_k$ only play a role later, when we have to define the function $\psi$ on the plane which equals $u'$ on the graph of $u$: it becomes at that point necessary for us to have  a gap between the sets $V_k$, where we shall stipulate the value of $u_k'$, and the sets $W_k$ where we permit some non-zero addition to $u_{k-1}'$ in the definition of $u_{k}'$.  Until Lemma~\ref{psilemma}, however, little is lost if one does not distinguish between $V_k$ and $W_k$.

 Note for $1 \leq i \leq n_k$,  $x \in V_k^i$ and $y \notin V_{k-1}$, we in fact have, choosing $z \in E \cap V_k^i$, that $| z - x|\leq r_k^i \leq \lambda(V_k)$, and so by~\eqref{measvk},
$$
 | x- y|  \geq |y - z| - |z - x|  \geq  \delta_{k-1} - \lambda(V_k) \geq \delta_{k-1}/2.
$$
Thus, since this holds for all $1 \leq i \leq n_k$,
\begin{equation} \label{distvkvk-1}
\mathrm{dist}(V_k, \mathbb{R} \setminus V_{k-1}) \geq \mathrm{dist}(E,\mathbb{R}\setminus V_{k-1}) / 2.
\end{equation}
\begin{lemma}
\label{uuklemma}
There exist a strictly increasing function $u \in C^{\infty}([a,b] \setminus E) \cap C([a,b])$ and a sequence $\{u_k\}_{k=0}^{\infty}$ of strictly increasing functions $u_k \in C^{\infty}([a,b])$ such that, for all $k \geq 0$,
\begin{enumerate}[label=(\thelemma.\arabic*), leftmargin=*]
\item $u(a) = u_k (a)$ and $u(b) = u_k (b)$; \label{ukadmissible}
\item $u(x) = u_k (x)$ for all $x \in [a,b] \setminus W_k$, and consequently $u'(x) = u_k'(x)$ for all $x \in [a,b]\setminus \overline{W_k}$; \label{u=ukoffvk}
\item $u' (x) \geq h_k$ for all $x \in V_k \setminus E$, and $u_k'(x) \geq h_l$ for all $x \in V_l$ for all $0 \leq l \leq k$; \label{psibigonvk}
\item $u'(x) = u_k'(x)$ for all $x \in [a,b] \setminus V_k$, and $u_k'(x) \leq h_k + 2$ for all $ x\in [a,b]$; and \label{u'offvk}
\item $u_k \to u$ uniformly on $[a,b]$. \label{unif}
\end{enumerate}
\end{lemma}
\begin{proof}
We first exhibit a sequence $\{\phi_k\}_{k=0}^{\infty}$ of functions $\phi_k \in C^{\infty}(\mathbb{R})$ such that for all $k\geq 0$,
\begin{enumerate}[label=(\thelemma.\alph*), leftmargin=*]
\item $ 1 \leq \phi_k (x) \leq h_k +2$ for all $x \in \mathbb{R}$; \label{psikbds}
\item $h_k + 1 \leq \phi_k(x)$ for all $x \in V_k$; \label{psikbigonvk}
\item $ \phi_k (x)= \phi_{l}(x)$ for  all $x \in \mathbb{R} \setminus W_{l}$ for all $0 \leq l \leq k$; \label{psiloffvk}
\item $h_l \leq \phi_k (x)$ for  $x \in V_l$ for all $0 \leq l \leq k$; and \label{psilbigonvk}
\item $\int_{W_l^i \cap (a,b)} \phi_k = \int_{W_l^i \cap (a,b)} \phi_l$ for all $1 \leq i \leq n_l$ and all $0 \leq l \leq k$. \label{intpsil}
\end{enumerate}
Define $\phi_0 (x) = h_0 +1$ for all $x \in \mathbb{R}$, which clearly satisfies~\ref{psikbds}--\ref{intpsil}.   Let $k \geq 1$, and consider $1 \leq j \leq n_{k-1}$.  Note that inequalities~\eqref{measvk} and~\eqref{centred} imply that
\begin{align*}
\lambda(W_{k-1}^j \cap V_k \cap(a,b)) \leq \lambda(V_k) \leq \frac{ \mathrm{dist}(E, \mathbb{R} \setminus W_{k-1})}{2(h_k - h_{k-1} + 1)} & \leq \frac{\tilde{r}_{k-1}}{2(2(h_k - h_{k-1}) +1)} \\
&\leq \frac{\lambda(W_{k-1}^j \cap (a,b))}{2(h_k - h_{k-1}) + 1},
\end{align*}
and so
$$
\frac{\lambda(W_{k-1}^j \cap (a,b))}{\lambda(V_k \cap W_{k-1}^j \cap (a,b))} \geq 2 (h_k - h_{k-1}) +1 > h_k - h_{k-1} + 1.
$$
Hence we can choose $\rho_k \in C^{\infty}(\mathbb{R})$ such that
 \begin{gather}
 \rho_k (x) = 0 \ \textrm{for all $x \in \mathbb{R} \setminus W_{k-1};$} \label{phikoffvk-1}\\
 -1 \leq \rho_k (x) \leq h_k - h_{k-1}\ \textrm{for all $x \in \mathbb{R}$};\label{phikbds}\\
 \rho_k (x) = h_k - h_{k-1}\ \textrm{for all  $x \in V_k$; and} \label{phikonvk} \\
 \int_{W_{k-1}^j \cap(a,b)} \rho_k = 0 \ \textrm{for each $1 \leq j \leq n_{k-1}$}. \label{intphik=0}
 \end{gather}
For example, fix $1\leq j \leq n_{k-1}$, and note that when considering open sets $G_k^j$, $\tilde{G}_k^j$ such that
$$
V_k \cap W_{k-1}^j \Subset G_k^j \Subset \tilde{G}_k^j \Subset W_{k-1}^j,
$$
the value $\lambda(\tilde{G}_k^j \cap(a,b))/ \lambda(G_k^j \cap(a,b))$ depends continuously on the measures of the two sets $\tilde{G}_k^j \cap(a,b)$ and $G_k^j \cap(a,b)$, and takes values greater than but arbitrarily close to $1$, and less than but arbitrarily close to $\lambda( W_{k-1}^j \cap(a,b)) / \lambda(W_{k-1}^j \cap V_k \cap(a,b)) > h_k - h_{k-1} + 1$.  Thus we may choose sets $\tilde{G}_k^j$ and $G_k^j$ such that
$$
\frac{\lambda(\tilde{G}_k^j \cap(a,b))}{\lambda(G_k^j \cap(a,b))} = h_k - h_{k-1}  + 1;
$$
that is
$$
(h_k - h_{k-1}) \lambda(G_k^j \cap(a,b)) = \lambda((\tilde{G}_k^j \cap (a,b)) \setminus (G_k^j \cap (a,b))).
$$
Then defining $\rho_k^j \colon \mathbb{R} \to \mathbb{R}$ by
$$
\rho_k^j (x) =
\begin{cases}
h_k - h_{k-1} & x \in G_k^j \\
-1 & x \in \tilde{G}_k^j \setminus G_k^j\\
0 & \mathrm{otherwise},
\end{cases}
$$
we see that
$$
\int_{a}^{b}\rho_k^j = \int_{\tilde{G}_{k-1}^j \cap (a,b)} \rho_k^j = (h_k - h_{k-1})\lambda(G_k^j \cap(a,b)) - \lambda((\tilde{G}_k^j \cap( a,b))\setminus (G_k^j \cap (a,b))) = 0.
   $$
Choosing an appropriate mollification, we can assume that $\rho_k^j$ is of class $C^{\infty}(\mathbb{R})$, the same equation holds, and that $\rho_k^j$ satisfies~\eqref{phikoffvk-1}--\eqref{intphik=0}, with $W_{k-1}^j \cap V_k$ replacing $V_k$ in condition~\eqref{phikonvk}.  Then defining $\rho_k = \sum_{j=1}^{n_{k-1}}\rho_k^j$ gives $\rho_k$ as claimed.

Using this $\rho_k$, we now suppose $\phi_{k-1}$ to be defined, and set
$
\phi_k = \phi_{k-1} + \rho_k
$.
This defines our sequence $\{\phi_k\}_{k=0}^{\infty}$.  We now show by induction on $k \geq 0$ that these functions satisfy the requirements~\ref{psikbds}--\ref{intpsil}.  Let $k \geq 1$, and suppose $\phi_{k-1}$ has been constructed in this way and satisfies all the conditions.

By~\eqref{phikoffvk-1} we see that $\phi_k = \phi_{k-1}$ off $W_{k-1}$, which gives~\ref{psiloffvk} by inductive hypothesis  and since $\{W_k\}_{k=0}^{\infty}$ is a decreasing sequence.  Then for points not in $W_{k-1}$, we see that the inequality in~\ref{psikbds} holds by inductive hypothesis~\ref{psikbds} and since $\{h_k\}_{k=0}^{\infty}$ is an increasing sequence.  For $x \in W_{k-1}$ we have, by inductive hypothesis~\ref{psikbigonvk},~\eqref{phikbds}, and inductive hypothesis~\ref{psikbds}, that
$$
1 \leq h_{k-1} \leq  \phi_{k-1} (x)  -1 \leq \phi_k (x) \leq \phi_{k-1}(x)  + (h_k - h_{k-1}) \leq  h_k + 2.
$$
Hence the inequality in~\ref{psikbds} holds everywhere, as required.
Note that for $x \in V_k$ we have by~\eqref{phikonvk} and inductive hypothesis~\ref{psikbigonvk}, since $V_k \subseteq V_{k-1}$, that
$$
\phi_k(x) = \phi_{k-1}(x) + h_k - h_{k-1} \geq 
 h_k + 1,
$$
as required for~\ref{psikbigonvk}.  This implies~\ref{psilbigonvk} when $x \in V_k$.  Otherwise, choose the greatest index $0 \leq l < k$ such that $x \in V_l$.   If $l < k-1$, then $x \notin V_{k-1}$, so inequality~\ref{psilbigonvk} follows by~\eqref{phikoffvk-1} and inductive hypothesis~\ref{psilbigonvk}.  If $l = k-1$, then $ x \in V_{k-1}$, and so by~\eqref{phikbds} and inductive hypothesis~\ref{psikbigonvk},
$$
\phi_k(x) \geq \phi_{k-1}(x) - 1 \geq h_{k-1},
$$
hence~\ref{psilbigonvk} holds in all cases.  For the claim~\ref{intpsil}, there is nothing to prove for $l = k$, so let $0 \leq l < k$,  and fix $0 \leq i \leq n_l$. Then using~\eqref{phikoffvk-1},~\eqref{intphik=0}, and the inductive hypothesis we have that
\begin{align*}
\int_{W_l^i \cap(a,b)}\phi_k = \int_{W_l^i \cap(a,b)} \phi_{k-1} + \int_{W_l^i \cap(a,b)} \rho_k  & = \int_{W_l^i \cap (a,b)} \phi_{k-1} + \int_{W_l^i \cap W_{k-1} \cap(a,b)} \rho_k \\
& = \int_{W_l^i \cap(a,b)} \phi_{k-1} \\
& = \int_{W_l^i \cap (a,b)} \phi_l,
\end{align*}
since $\{W_k\}_{k=0}^{\infty}$ is decreasing, so $W_l^i \cap W_{k-1} \cap (a,b)$ consists of components $W_{k-1}^j \cap (a,b)$ of $W_{k-1} \cap (a,b)$.

Using~\eqref{vkshrink} we see that for all $ x \notin E$ there is $k \geq 1$ such that $x \notin W_l$ for all $l \geq k$, thus by~\ref{psiloffvk} letting $\phi (x) = \lim_{k \to \infty} \phi_k(x)$ defines a well-defined function $\phi \in C^{\infty}(\mathbb{R} \setminus E)$ such that
\begin{equation}
\phi(x)  = \phi_k (x) \ \textrm{for all} \ x \notin W_k \label{psioffvk}.
\end{equation}
By~\ref{psikbds} we have that $\phi (x)  \geq 1$ for all $x \in \mathbb{R} \setminus E$. 

Now,  $|\phi_k| \leq |\phi_0| + \sum_{l=1}^{\infty}|\rho_l|$ for all $k \geq 0$, and using~\eqref{phikoffvk-1},~\eqref{phikbds}, and~\eqref{measvk} we see that
\begin{align*}
\int_a^b |\phi_0| + \int_a^b \sum_{l=1}^{\infty} |\rho_l| & \leq (b-a) (h_0 + 1) + \sum_{l=1}^{\infty} \lambda([a,b] \cap W_{l-1})(h_l - h_{l-1} +1) \\
& \leq  (b-a)( h_0 + h_1 + 2) + \sum_{l=1}^{\infty} (h_{l+1}+1) \lambda(V_l)\\
& \leq (b-a)(2h_1 + 2) + \sum_{l=1}^{\infty} 2^{-l} \\
& < \infty.
\end{align*}
So by the dominated convergence theorem $\phi \in L^1(a,b)$, and
\begin{equation} \label{intcvg}
\int_a^b \phi_k \to \int_a^b \phi\ \textrm{as}\ k \to \infty.
\end{equation}

We now define strictly increasing functions $u_k \in C^{\infty}([a,b])$ for each $k \geq 0$ and $u \in C^{\infty}([a,b] \setminus E) \cap  C([a,b])$ by
$$
u_k (x) = \int_a^x \phi_k (t) \, dt\ \text{and}\ u(x) = \int_a^x \phi(t)\, dt,
$$
and so $u_k' = \phi_k$ everywhere and $u' = \phi$ off $E$, in particular almost everywhere.  Condition~\ref{psibigonvk} follows immediately from~\ref{psilbigonvk}.  Condition~\ref{u'offvk} follows immediately from~\ref{psiloffvk} and~\ref{psikbds}.  Condition~\ref{ukadmissible} follows since by~\ref{intpsil} and~\eqref{intcvg}, 
$$
u_k(b) = \int_a^b \phi_k = \int_{W_0 \cap (a,b)} \phi_k = \int_{W_0 \cap (a,b)} \phi_0 = \int_{W_0 \cap (a,b)}\phi = u(b),
$$
and since clearly $u_k(a) = 0 = u(a)$ by definition.

Let $k \geq 0$, and suppose $ x \in [a,b] \setminus W_k$.  Then either we have $x \leq \tilde{a}_k^i$ for all $1 \leq i \leq n_k$, or  we have for some $1 \leq i_x \leq n_k$ that $\tilde{b}_k^{i_x} \leq x$ and $x \leq \tilde{a}_k^i$ for all $i_x < i \leq n_k$.  In the first case we see immediately that, since $[a,x] \cap W_k = \emptyset$,~\eqref{psioffvk} implies that
$$
u(x) = \int_a^x \phi(t)\, dt = \int_a^x \phi_k (t) \, dt = u_k (x).
$$
Otherwise we argue by~\eqref{intcvg},~\ref{intpsil}, and~\ref{psiloffvk} that
\begin{align*}
u(x) = \int_a^x \phi  = \sum_{i=1}^{i_x}\int_{W_k^i \cap (a,b)} \phi + \int_{[a,x] \setminus W_k} \phi   &= \sum_{i=1}^{i_x} \int_{W_k^i \cap (a,b)} \phi_k + \int_{[a,x] \setminus W_k} \phi_k  \\
& = \int_a^x \phi_k \\
& = u_k(x),
\end{align*}
as required for~\ref{u=ukoffvk}.

Fix $1 \leq i \leq n_k$, and let $ x \in W_k^i$.  Since $u$ and $u_k$ are increasing, using~\ref{u=ukoffvk},~\ref{psikbds}, and~\eqref{measvk} we see that
\begin{align*}
|u_k (x) - u(x)|  \leq u_k(\tilde{b}_k^i) - u(\tilde{a}_k^i) = u_k(\tilde{b}_k^i) - u_k(\tilde{a}_k^i) = \int_{\tilde{a}_k^i}^{\tilde{b}_k^i} \phi_k  & \leq (h_k + 2)(\tilde{b}_k^i - \tilde{a}_k^i) \\
& \leq (h_k + 2)\lambda(V_k) \\
& \leq 2^{-k}.
\end{align*}
Since $u= u_k$ off $W_k$, we then have that $\sup_{x \in [a,b]}|u_k (x) - u (x)| \leq 2^{-k}$, hence $u_k$ converges to $u$ uniformly, as required for~\ref{unif}.
\end{proof}
\section{Construction of the potential}
 The construction of our potential, $\Phi$,  is based on that  which constitutes the proof of Theorem~10 in~\citet{Csornyei-etal-2008}.   We construct a sequence of $C^{\infty}(\mathbb{R}^2)$ functions $\{\Phi^k\}_{k= 0 }^{\infty}$ which have steep gradients on open sets $\Omega_k$ around the graph $U(E)$ of $u$ on $E$.  Because we have fixed the derivative $\phi$ of our minimizer $u$ with which we have to compare the derivatives of $\Phi$, the sets $\Omega_k$ are now given before the construction.  This contrasts with the situation of~\citeauthor{Csornyei-etal-2008}, where the sets could be chosen small enough at each stage of the construction of the sequence.  We have of course carefully chosen $\Omega_k$, or more precisely in fact $V_k$, so that all the properties required at this stage hold with these fixed sets.  

Let $\Omega_0 = \mathbb{R}^2$, and for $k \geq 1$ and $1 \leq i \leq n_k$ define $\Omega_k^i = V_k^i \times u(V_k^i) = V_k^i \times u_k (V_k^i)$,  and $\Omega_k = \bigcup_{i=1}^{n_k} \Omega_k^i$.  So $\Omega_k$ is an open set satisfying $\Omega_k \Supset U(E)$.
\begin{lemma}
\label{omegalemma}
For this definition of the sequence $\{\Omega_k\}_{k=0}^{\infty}$, we have that
\begin{enumerate}[label=(\thelemma.\arabic*), leftmargin=*]
\item $\bigcap_{k=1}^{\infty}\Omega_k = U(E)$; \label{omegashrink}
\listpart{and for each $k \geq 1$,}
\item $\mathrm{dist}(\Omega_k, \mathbb{R}^2 \setminus \Omega_{k-1}) \geq \delta_{k-1}/2$; and \label{dOmktooffOmk-1}
\item $\sum_{i=1}^{n_k} \mathrm{diam}(\Omega_k^i) <  \lambda(V_k)(h_k + 2)$. \label{diamomega}
\end{enumerate}
\end{lemma}
\begin{proof}
The inclusion $U(E) \subseteq \bigcup_{k=1}^{\infty}\Omega_k$ is clear.  Let $(x, y) \notin U(E)$.  If $x \notin E$ then there exists $k \geq 1$ such that $x \notin V_k$, so $(x,y) \notin \Omega_k$.  Otherwise, $x \in E$ but $y \neq u(x)$.  Since $x \in E$, for all $k \geq 1$ there exists $1 \leq i_k \leq n_k$ such that $x \in V_k^{i_k}$. Since $|b_k^{i_k} - a_k^{i_k}| < \lambda(V_k) \to 0$, there exists $k \geq 1$ such that $|u(b_k^{i_k}) - u(a_k^{i_k})| < |y - u(x)|/2$.  If $y \in (u(a_k^{i_k}), u(b_k^{i_k}))$, then
  $$
  | y - u(x) | \leq | y - u(a_k^{i_k})| + | u ( a_k^{i_k} ) - u(x)| \leq 2| u(b_k^{i_k}) - u(a_k^{i_k})| < | y - u(x)|,
  $$
  which is a contradiction, so $y \notin (u(a_k^{i_k}), u(b_k^{i_k}))$.  Since $x \in (a_k^{i_k}, b_k^{i_k})$ and the components of $V_k$ are pairwise disjoint, this implies that $(x,y) \notin \Omega_k$.

Fix $k \geq 1$, and let $(x_1, y_1) \in \Omega_k^i$ for some $1 \leq i \leq n_k$, but $(x_2, y_2) \notin \Omega_{k-1}$ (the result~\ref{dOmktooffOmk-1} is trivial if $k =1$ and hence no such point exists).  There exists $1 \leq j \leq n_{k-1}$ such that $x_1 \in V_{k-1}^j$, since $\{V_k\}_{k=0}^{\infty}$ are decreasing. First we suppose that $x_2 \notin V_{k-1}^j$.  Then since there must exist at least one point between $x_1$ and $x_2$ which does not lie in $V_{k-1}$, equation~\eqref{distvkvk-1} implies that
$$
\|(x_1, y_1) - (x_2, y_2)\| \geq |x_1 - x_2| \geq \delta_{k-1}/2.
$$
Otherwise $x_1, x_2 \in V_{k-1}^j$.    Notice that by~\ref{psibigonvk} and~\eqref{distvkvk-1} we have
$$
|u(b_{k-1}^j) - u(b_k^i)| \geq h_{k-1} \delta_{k-1}/2 \ \text{and}\ |u(a_{k-1}^j) - u(a_k^i)| \geq h_{k-1} \delta_{k-1}/2.
$$
Since $y_1 \in (u(a_k^i), u(b_k^i))$ but $y_2 \notin (u(a_{k-1}^j), u(b_{k-1}^j))$, this implies that
$$
\|(x_1, y_1) - (x_2, y_2)\| \geq |y_1 - y_2| \geq h_{k-1} \delta_{k-1}/2 \geq \delta_{k-1}/2,
$$
as required for~\ref{dOmktooffOmk-1}.

Finally, for $1 \leq i \leq n_k$,  we easily see using~\ref{u'offvk} that
$$
\mathrm{diam}(\Omega_k^i) \leq   |u_k(b_k^i) - u_k(a_k^i)| \leq (h_k + 2) \lambda(V_k^i),
$$
and hence, since $\{V_k^i\}_{i=1}^{n_k}$ are pairwise disjoint,
$$
\sum_{i=1}^{n_k} \mathrm{diam}(\Omega_k^i) \leq (h_k + 2)\sum_{i=1}^{n_k} \lambda(V_k^i) = (h_k + 2) \lambda(V_k),
$$
as required for~\ref{diamomega}.
\end{proof}
The final step before we construct the potential is to lift our derivative $\phi$ from the real line into the plane, i.e.\ to construct a function $\psi$ on the plane with which we can compare the potential, and which agrees with $\phi$ where necessary, i.e.\ on the graph of $u$.
\begin{lemma}
\label{psilemma}
There exists $\psi \in C^{\infty}(\mathbb{R}^2 \setminus U(E))$ such that
\begin{enumerate}[label=(\thelemma.\arabic*), leftmargin=*]
\item $\psi(x, y) \leq h_k + 2 $ for all $(x, y) \notin \overline{\Omega_{k}}$; \label{psi-1}
\item $\psi(x, y) \geq h_k $ for all $(x, y) \in \overline{\Omega_k} \setminus U(E)$; and \label{psi-2}
\item $\psi(x, u(x)) = u'(x)$ for all $x \in [a,b]\setminus E$. \label{psi-3}
\end{enumerate}
\end{lemma}
\begin{proof}
We construct a sequence $\{ \psi_k\}_{k=0}^{\infty}$ of functions $\psi_k \in C^{\infty}(\mathbb{R}^2)$ such that for $k \geq 0$,
\begin{enumerate}[label=(\thelemma.\alph*),leftmargin=*]
\item $\psi_k(x, y) \leq h_k + 2$ for all $(x, y) \in \mathbb{R}^2$; \label{psik-1}
\item $\psi_k(x, u_k(x)) = u_k'(x)$ for all $x \in [a,b]$; \label{psik-3}
\item $\psi_k(x, y) = \phi_k(x)$ for all $(x, y) \in \Omega_k$, where $\phi_k$ is as constructed in the proof of Lemma~\ref{uuklemma}; \label{psik-4}
\listpart{and for $k \geq 1$,}
\item $\psi_k (x, y) = \psi_{k-1}(x,y) $ for all $(x, y) \notin \overline{\Omega_{k-1}}$; and\label{psik-0}
\item $\psi_k(x, y) \geq h_{k-1}$ for all $(x,y) \in \Omega_{k-1}$. \label{psik-2}
\end{enumerate}
Defining $\psi_0 = h_0 + 1$ satisfies all the conditions~\ref{psik-1}--\ref{psik-4}.  Suppose $\psi_{k-1}$ has been constructed as required, for $k \geq 1$.  It is at this point that the positive distance between $W_k$ and $\mathbb{R} \setminus V_k$ becomes useful.  We define a new sequence of open sets $\{\tilde{\Omega}_k\}_{k=0}^{\infty}$ in $\mathbb{R}^2$ such that $U(E) \subseteq \tilde{\Omega}_k \subseteq \Omega_k$ and $\tilde{\Omega}_k \supseteq \Omega_{k+1}$ by setting $\tilde{\Omega}_0 = \mathbb{R}^2$, and for $k \geq 1$,  $\tilde{\Omega}_k^i = W_k^i \times u(W_k^i) = W_k^i \times u_k (W_k^i)$, and $\tilde{\Omega}_k = \bigcup_{i=1}^{n_k} \tilde{\Omega}_k^i$.  Choose a function $\pi_k \in C^{\infty}(\mathbb{R}^2)$ such that $0 \leq \pi_k \leq 1$ on $\mathbb{R}^2$, $\pi_k = 0$ off $\Omega_{k -1}$, and $\pi_k = 1$ on $\tilde{\Omega}_{k-1}$.  Using $\phi_k \in C^{\infty}(\mathbb{R})$  from the proof of Lemma~\ref{uuklemma}, we define
$$
\psi_k(x, y) = \psi_{k-1}(x, y) + \pi_k (x, y) ( \phi_k(x) - \psi_{k - 1}(x, y)).
$$
Condition~\ref{psik-0} is immediate.  Since $\Omega_k \subseteq \tilde{\Omega}_{k-1}$, we see that $\psi_k (x,y) = \phi_k(x)$ for $(x,y) \in \Omega_k$, as required for~\ref{psik-4}.

We note that by inductive hypothesis~\ref{psik-1} and~\ref{psikbds},
$$
\psi_k = ( 1- \pi_k) \psi_{k-1} + \pi_k \phi_k \leq ( 1- \pi_k) (h_{k-1} + 2) + \pi_k (h_k + 2)
 \leq h_k + 2,
$$
since $\{h_k\}_{k=0}^{\infty}$ are increasing, as required for~\ref{psik-1}.  Now let $(x, y) \in \Omega_{k-1}$, to check~\ref{psik-2}.  Using inductive hypothesis~\ref{psik-4} and~\ref{psilbigonvk} we see that
$$
\psi_k= ( 1- \pi_k) \psi_{k-1}+ \pi_k \phi_k = ( 1- \pi_k) \phi_{k-1} + \pi_k \phi_k  \geq (1 - \pi) h_{k-1} + \pi h_{k-1}  = h_{k-1},
$$
as required.

For~\ref{psik-3} we need to consider cases.  First suppose $x \notin V_{k-1}$, so $(x, u_k(x)) \notin \Omega_{k-1}$.  Then by~\ref{u=ukoffvk} and inductive hypothesis~\ref{psik-3}
$$
\psi_k (x, u_k(x)) = \psi_{k-1}(x, u_k (x)) = \psi_{k-1}(x, u_{k-1}(x))  = u_{k-1}'(x)= u_k'(x),
$$
as required.  For $x \in W_{k-1}$, we see that then $(x, u_k(x)) \in \tilde{\Omega}_{k-1}$, and so
$$
\psi_k(x, u_k(x)) = \psi_{k-1} (x, u_k(x)) + u_k'(x) - \psi_{k-1}(x, u_k(x)) = u_k'(x),
$$
as required.  The final case is for $x \in V_{k-1} \setminus W_{k-1}$, in which case we argue that by~\ref{u=ukoffvk} and inductive hypothesis~\ref{psik-3},
\begin{align*}
\psi_k(x, u_k(x))
& = \psi_{k-1}(x, u_k(x)) + \pi_k (x, u_k(x)) ( u_k'(x) - \psi_{k-1}(x, u_k(x))) \\
& = \psi_{k-1}(x, u_k(x)) + \pi_k(x, u_k(x))(u_{k-1}'(x) - \psi_{k-1}(x, u_{k-1}(x)))\\
& = \psi_{k-1}(x, u_{k-1}(x))\\
& = u_{k-1}'(x) \\
& = u_k'(x),
\end{align*}
as required.  Hence the result in general.

Let $(x,y) \in \mathbb{R}^2 \setminus U(E)$.  By~\ref{omegashrink} there exists $k \geq 1$ such that $(x,y) \notin \overline{\Omega_{k-1}} \setminus \overline{\Omega_k}$.  Then~\ref{psik-0} implies that $\lim_{l \to \infty} \psi_l $ exists and equals $\psi_k$ on an open set around $(x,y)$.  Hence $\psi_k$ converges to a function $\psi \in C^{\infty}(\mathbb{R}^2 \setminus U(E))$ such that $\psi = \psi_k$ on $\overline{\Omega_{k-1}} \setminus \overline{\Omega_k}$.  Condition~\ref{psi-1} follows from~\ref{psik-1}, and condition~\ref{psi-2} follows from~\ref{psik-2}.  For~\ref{psi-3}, we let $x \in [a,b] \setminus E$, find $k\geq 1$ such that $x \in \overline{V_{k-1}} \setminus \overline{V_k}$ and hence that $(x, u(x)) \in \overline{\Omega_{k-1}} \setminus \overline{\Omega_k}$, and use~\ref{u=ukoffvk} and~\ref{psik-3} to see that
$$
\psi(x, u(x)) = \psi_k(x, u(x)) = \psi_k(x, u_k(x)) = u_k'(x) = u'(x),
$$
as required.
\end{proof}
We now state and prove appropriate versions of Lemmas~12 and~13 in~\citet{Csornyei-etal-2008}.  For two vectors $x, y \in \mathbb{R}^2$, we write $[x, y]$ to denote the line segment in $\mathbb{R}^2$ connecting them.
\begin{lemma}
\label{lemma12'}
Let $\tau > 0$, $ e \in \mathbb{R}^2 \setminus \{0\}$, and suppose $\Omega \subseteq \mathbb{R}^2$ is an  open set such that $\Omega = \bigcup_{i=1}^{\infty} \Omega_i$ such that $\sum_{i=1}^{\infty} \mathrm{diam}(\Omega_i) < \frac{\tau}{2 \|e\|^2}$.

Then there exists $f \in C^{\infty}(\mathbb{R}^2)$ such that
\begin{itemize}
\item $0 \leq f(x) \leq \frac{\tau}{\|e\|}$ for all $ x \in \mathbb{R}^2$;
\item $ \mathrm{dist}(\nabla f(x) , [0, e]) < \tau$ for all $x \in \mathbb{R}^2$; and
\item $\|\nabla f(x) - e \| < \tau$ for $x \in \overline{\Omega}$.
\end{itemize}
\end{lemma}
\begin{proof}
We first show that it suffices to prove the result for $e = (1,0)$.  For an arbitrary $e \in \mathbb{R}^2$, find a rotation $R \colon \mathbb{R}^2 \to \mathbb{R}^2$ such that $\|e\|^{-1}Re = (1, 0)$.  Then
$$
\sum_{i=1}^{\infty} \mathrm{diam}( \|e\| R (\Omega_i)) = \|e\| \sum_{i=1}^{\infty} \mathrm{diam} ( \Omega_i) < \frac{\tau}{2 \|e \|},
$$
so $\|e\|R(\Omega)$ satisfies the assumptions for $\tilde{e} \mathrel{\mathop:}= (1,0)$ and $\tilde{\tau} \mathrel{\mathop:}= \frac{\tau}{\|e\|}$.

So by assumption there exists $\tilde{f} \in C^{\infty}(\mathbb{R}^2)$ satisfying the three conclusions for $\tilde{e}$ and $\tilde{\tau}$.  Define  $f \in C^{\infty}(\mathbb{R}^2)$ by $f(x) = \tilde{f}(\|e\|Rx)$.  Fix $ x \in \mathbb{R}^2$.  Then firstly
$$
0 \leq f(x) = \tilde{f}(\|e\|Rx) \leq \frac{\tilde{\tau}}{\|(1,0)\|}  = \frac{\tau}{\|e\|}.
$$
By assumption there exists $s \in [0,1]$ such that
$$
\left\| \nabla \tilde{f}(\|e\|Rx) - s(1,0)\right\| < \tilde{\tau}.
$$
Then for this $s \in [0,1]$ we have
\begin{align*}
\|\nabla f (x) - s e\| & = \left\|\|e\| \nabla \tilde{f}(\|e\| Rx)R - s\|e\|\|e\|^{-1}eR^{-1}R \right\| \\
& = \left\| \|e\|(\nabla \tilde{f}(\|e\|Rx) - s\|e\|^{-1}eR^{-1})\right\| \\
& =\|e\| \left\| \nabla \tilde{f}(\|e\|Rx) - s(1,0)\right\| \\
& < \|e\| \tilde{\tau} \\
& = \tau.
\end{align*}
Now let $ x \in \overline{\Omega}$.  Then $\|e\|^{-1}Rx \in \|e\|^{-1}R \overline{\Omega}$, so by assumption we have that
$$
\left\|\nabla \tilde{f} (\|e\|^{-1}Rx) - (1,0)\right\| < \tilde{\tau}.
$$
Thus
\begin{align*}  \| \nabla f (x) - e \| & = \left\| \|e\| \nabla \tilde{f}(\|e\| Rx)R - \|e\| \|e\|^{-1}  eR^{-1}R\right\| \\
& = \|e\| \left\| \nabla \tilde{f} (\|e\| Rx) - (1,0)\right\| \\
& < \|e\| \tilde{\tau} \\
& = \tau,
\end{align*}
as required.

So we can indeed assume without loss of generality that $e = (1,0)$.  By expanding each $\Omega_k^i$ slightly so that the inequality $\sum_{i=1}^{\infty} \mathrm{diam}( \Omega_k^i) < \tau/2$ is retained, and using a suitable mollification, it suffices to construct a Lipschitz function $g \colon \mathbb{R}^2 \to \mathbb{R}$ such that
\begin{itemize}
\item $0 \leq g (x) \leq \tau/2$ for all $x \in \mathbb{R}^2$;
\item $g_x (x) \in [0,1]$ and $g_y (x) \in [ -\frac{\tau}{2},  \frac{\tau}{2}]$ for every $x \in \mathbb{R}^2$; and
\item $g_x (x) =1 $ for $x \in \Omega$.
\end{itemize}
To do this, we first note that for any function $\gamma \colon \mathbb{R} \to \mathbb{R}$,
$$
\lambda( \{ s \in \mathbb{R} : (s, \gamma (s)) \in \Omega\}) \leq \sum_{i=1}^{\infty}\lambda ( \{s \in \mathbb{R}: (s, \gamma(s)) \in  \Omega_i\}) \leq \sum_{i=1}^{\infty} \mathrm{diam}(\Omega_i) < \frac{\tau}{2}.
$$
In particular, defining $g \colon \mathbb{R}^2 \to \mathbb{R}$ just as in~\citet{Csornyei-etal-2008}, by
$$
g(x,y) = \sup \left\{ x - b + \int_{\{ s \in \mathbb{R} : (s, \gamma (s)) \in \Omega\}} \left( 1 -  \frac{2 |\gamma'(s)|}{\tau}\right)\, ds \right\},
$$
where the supremum is taken over all $b \in \mathbb{R}$ such that $b \geq x$ and all $ \gamma \colon (-\infty, b] \to \mathbb{R}$ such that $\mathrm{Lip}(\gamma) < \frac{2}{\tau}$ and $\gamma(b) = y$, satisfies the requirements just as proved in~\citet{Csornyei-etal-2008}.
\end{proof}
\begin{lemma}
\label{lemma13'}
Let $\epsilon > 0$, $e^0, e^1 \in \mathbb{R}^2$ be distinct vectors, and $\Omega \Subset \Omega' \subseteq \mathbb{R}^2$ be open sets such that  $\Omega = \bigcup_{i=1}^{\infty} \Omega_i$ such that
$$
\sum_{i=1}^{\infty} \mathrm{diam} ( \Omega_i) < \frac{\epsilon/ (2 \|e^0 - e^1\|^2)}{(1 + ( \|e^0 - e^1\| \delta)^{-1})},
$$
where $0 < \delta \leq \mathrm{dist}(\Omega, \mathbb{R}^2 \setminus \Omega')/2$.  Let $g^0 \in C^{\infty}(\mathbb{R}^2)$.

Then there exists $g^1 \in C^{\infty}(\mathbb{R}^2)$ such that
\begin{itemize}
\item $\|g^1 - g^0\|_{\infty} < \frac{\epsilon}{ \|e^0 - e^1\|}$;
\item $g^1 = g^0$ off $\Omega'$;
\item $\mathrm{dist}(\nabla g^1 (x), [e^0, e^1]) < \epsilon + \| \nabla g^0 (x)- e^0\|$ for all $x \in \mathbb{R}^2$; and
\item $\| \nabla g^1 (x) - e^1 \| < \epsilon + \| \nabla g^0 (x)- e^0 \|$ for all $x \in \overline{\Omega}$.
\end{itemize}
\end{lemma}
\begin{proof}
Let $\tau \mathrel{\mathop:}= \frac{\epsilon}{1+ (\delta \|e^0 - e^1\|)^{-1}}$ and apply Lemma~\ref{lemma12'} with this $\tau$, the set $\Omega$ as given, and vector $e \mathrel{\mathop:}= e^1 - e^0$.  Let $f \in C^{\infty}(\mathbb{R}^2)$ be the resulting function.

Choose $ \chi\in C^{\infty}(\mathbb{R}^2)$ such that $0 \leq \chi \leq 1$, $ \chi=1$ on $\Omega$, and $\chi=0$ off $\Omega'$, and $\| \nabla \chi \| \leq \delta^{-1}$.  Define $g^1 = g^0 + \chi f$.  Clearly $g^1 \in C^{\infty}(\mathbb{R}^2)$. We see immediately from Lemma~\ref{lemma12'} that
$$
\|g^1 - g^0\|_{\infty} \leq \|f\|_{\infty} \leq \frac{\tau}{ \|e^1 - e^0\|} < \frac{\epsilon}{ \|e^1 - e^0\|}.
$$
Also note, by the properties of $\chi$ and Lemma~\ref{lemma12'},  we have for $x \in \mathbb{R}^2$ that
\begin{align*}
\mathrm{dist}(\nabla g^1(x), [e^0, e^1]) & \leq \mathrm{dist} (\nabla g^0 (x)- e^0 + (\chi \nabla f) (x)+ (f \nabla \chi)(x), [0,e^1 - e^0]) \\
& \leq \| \nabla g^0(x) - e^0 \| + \mathrm{dist}( \nabla f (x), [0,e]) + \| (f \nabla \chi)(x)\|\\
& \leq \| \nabla g^0(x) - e^0 \| + \tau + \frac{\tau}{ \delta\|e^0 - e^1\|}\\
& = \| \nabla g^0 (x)- e^0 \| + \epsilon.
\end{align*}
For $x \in \overline{\Omega}$ we have, since $\chi(x) = 1$, that
\begin{align*}
\| \nabla g^1 (x)- e^1\| & \leq \| \nabla g^0 (x)- e^0\| + \| (\chi \nabla f )(x)- (e^1 - e^0)\| + \| (f \nabla \chi)(x)\| \\
& \leq \| \nabla g^0 (x)- e^0\| + \| \nabla f (x)- e\| + \frac{\tau}{ \delta\|e^0 - e^1\|}\\
& \leq \| \nabla g^0 (x)- e^0\| + \tau ( 1 + \frac{1}{\delta\|e^0 - e^1\|}) \\
& = \| \nabla g^0 (x)- e^0\| + \epsilon.\qedhere
\end{align*}
\end{proof}
We now construct a potential  $\Phi$ which will  satisfy the conditions of Lemma~\ref{lemma11'} with the function $\psi$ given by Lemma~\ref{uuklemma}.  We now assign values to our increasing sequences, and define two new sequence $\{t_k\}_{k = 1}^{\infty}$ and $\{B_k\}_{k=0}^{\infty}$ by setting for $k \geq 0$:
\begin{itemize}
\item $h_k = 10 (3 +  2^{k+1})$;
\item $t_k = 3 + 2^k$;
\item $B_k = 4 + 320\omega'(h_{k+1} + 2)$, and $A_k = 3 t_k B_k$.
\end{itemize}
Also for $k \geq 0$, define numbers $\eta_k = 1-2^{-k-1}> 0$ and $\epsilon_k = 2^{-k}(4n_k)^{-1} > 0$, and vector $e_k = (-A_k, B_k) \in \mathbb{R}^2$.  We inductively construct a sequence $\{\Phi^k\}_{k=0}^{\infty}$ of functions $\Phi^k \in C^{\infty}(\mathbb{R}^2)$ satisfying, for $ k \geq 0$,
\begin{gather}
 \| \nabla \Phi^k (x,y)- e_k\| < \eta_k \ \textrm{for all $(x,y) \in \overline{\Omega_k}$;} \label{25'}\\
 \shortintertext{and for $k \geq 1$,}
 \|\Phi^k - \Phi^{k-1} \|_{\infty} < \epsilon_{k-1}; \label{27'}\\
 \Phi^k = \Phi^{k-1} \ \textrm{off $\Omega_{k-1}$; and} \label{28'}\\
 \mathrm{dist} (\nabla \Phi^k (x,y), [e_{k-1}, e_k]) < \eta_k \ \textrm{for all $(x,y) \in  \overline{\Omega_{k-1}}.$} \label{29'}
\end{gather}
We define $\Phi^0(x,y) = -A_0x + B_0y$, which clearly satisfies~\eqref{25'}.  Suppose for $k \geq 1$ that we have constructed $\Phi^{k-1}$ as claimed. To construct $\Phi^k$ we apply Lemma~\ref{lemma13'} with data $\epsilon = \epsilon_{k-1}$, $e^0 = e_{k-1}$, $e^1 = e_k$, $\Omega = \Omega_k$, $\Omega' = \Omega_{k-1}$,  $\delta = \delta_{k-1}/4$, and $g^0 = \Phi^{k-1}$.  We must check that the assumptions of the Lemma hold with these values.
First recall that~\ref{dOmktooffOmk-1} gives that $\delta_{k-1}/{4} < \mathrm{dist} ( \Omega_k, \mathbb{R}^2 \setminus \Omega_{k-1})/2 $ indeed.   We see by~\ref{diamomega} and~\eqref{measvk} that
\begin{align*}
\sum_{i=1}^{n_k}\mathrm{diam}( \Omega_k^i)
& \leq  \lambda (V_k) (h_k + 2) \\
& < \frac{2^{-(k-1)}(4n_{k-1})^{-1}/ 2 A_k^2}{ 1 + 4\delta_{k-1}^{-1}}\\
 & < \frac{\epsilon_{k-1}/2 \|e^k - e^{k-1}\|^2}{ 1+ 4 ( \delta_{k-1} \|e^k - e^{k-1}\|^2)^{-1}}.
\end{align*}
We define $\Phi^k$ as the function $g^1$ given by the Lemma.

Then~\eqref{28'} is immediate, and since $\|e_k - e_{k-1}\| \geq 1$, we see that $\| \Phi^k - \Phi^{k-1}\|_{\infty} < \epsilon_{k-1}$ as required for~\eqref{27'}.  For~\eqref{29'} we let $(x,y) \in \overline{\Omega_{k-1}}$ and use inductive hypothesis~\eqref{25'} and the properties given by Lemma~\ref{lemma13'} to see that
\begin{align*}
\mathrm{dist}(\nabla \Phi^k (x,y) , [e_{k-1}, e_k]) & < \epsilon_{k-1} + \| \nabla \Phi^{k-1}(x,y) - e_{k-1}\| \\
& \leq \epsilon_{k-1} + \eta_{k-1} \\
& \leq 2^{-(k-1 + 2)} + 1- 2^{-k}\\
& = \eta_k.
\end{align*}
Similarly for~\eqref{25'}, we let $(x,y) \in \overline{\Omega_{k}}$ and use the Lemma and the inductive hypothesis~\eqref{25'} again, noting that $\Omega_k \subseteq \Omega_{k-1}$, to see that
$$
\| \nabla \Phi^k (x,y) - e_k\|_{2}  < \epsilon_{k-1} + \| \nabla \Phi^{k-1} (x,y) - e_{k-1}\|_{2}  < \eta_k.
$$
Hence we can construct such a sequence $\{\Phi^k\}_{k=0}^{\infty}$ as claimed.  We now check that this gives us the potential we require for Lemma~\ref{lemma11'}, with $S = U(E)$. By~\eqref{27'} and since $\epsilon_k \leq 2^{-(k+2)}$, we see that $\Phi^k$ converge uniformly to some $\Phi \in C(\mathbb{R}^2)$.

Fix $(x,y) \in \mathbb{R}^2 \setminus (U(E))$. By~\ref{omegashrink} there is $k\geq 1$ such that $(x,y) \in \overline{\Omega_{k-1}} \setminus \overline{\Omega_k}$, and hence $\Phi \in C^{\infty}(\mathbb{R}^2 \setminus (U(E))) \cap C(\mathbb{R}^2)$ and $\nabla \Phi = \nabla \Phi^l$ on $\overline{\Omega_{k-1}} \setminus \overline{\Omega_k}$, for all $l \geq k$, by~\eqref{28'}.  Moreover, by~\eqref{29'},
\begin{gather*}
\Phi_y (x,y) = \Phi_y^k (x,y) \geq B_{k-1} - \eta_{k-1} \geq B_0 - 1 \geq 3
\end{gather*}
as required for the second inequality of~\ref{11beta}.  More precisely, by~\eqref{29'} there is $ s \in [0,1]$ such that $\| \nabla \Phi (x,y) - (s e_{k-1} + (1-s) e_k)\|_{2} < 1$.  Using this we see that
\begin{align*}
-\Phi_x (x,y) 
& \leq s A_{k-1} + ( 1- s) A_k + 1 \\
& \leq 3 t_k ( s B_{k-1} + (1 - s) B_k ) + 1 \\
& \leq 3 t_k ( \Phi_y (x,y) + 1) + 1 \\
& \leq 5 t_k \Phi_y (x,y),
\end{align*}
thus $ (-\Phi_x / \Phi_y) (x,y) \leq 5 t_k$.
Similarly
\begin{align*}
-\Phi_x (x,y) 
& \geq s A_{k-1} + (1-s) A_k - 1\\
& \geq 3 t_{k-1}( s B_{k-1} + (1-s)B_k) - 1\\
& \geq 3 t_{k-1} (\Phi_y (x,y)- 1) -1 \\
& \geq t_{k-1} \Phi_y (x,y),
\end{align*}
thus $ (-\Phi_x  / \Phi_y ) (x,y) \geq t_{k-1}$.  Condition~\ref{11beta} follows since $t_{k-1} \geq t_0 = 4$.    We know from~\ref{psi-2} and~\ref{psi-1} that $h_{k-1} \leq \psi(x,y) \leq h_k + 2$.  Thus, by properties of $\omega$,
$$
\Phi_y (x,y)
> B_{k-1} -1 =  3 + 320 \omega' (h_k + 2) \geq 320 \omega'(\psi(x,y)),
$$
as required for~\ref{11gamma}.  

We note that, from the definitions,
\begin{gather*}
10t_k = 10 (3 + 2^{k} ) =   h_{k-1},
\shortintertext{and}
h_k + 2 = 10 ( 3 + 2^{k+1})  + 2 \leq  10 \cdot 2^4 (3 + 2^{k-1}) = 160 t_{k-1}.
\end{gather*}
So  we see that
$$
-2\Phi_x (x,y) / \Phi_y (x,y)  \leq 10 t_k \leq \psi(x,y)  \leq h_k +2 \leq -160 \Phi_x (x,y) / \Phi_y (x,y),
$$
and hence get~\ref{11delta}. 

We finally check~\ref{11epsilon}, so  let $\tilde{u} \in \mathrm{AC}(a,b)$.  The set $\tilde{U}^{-1}(U(E)) \subseteq E$ and is therefore null.  Fix $k \geq 1$, and note that by~\eqref{27'} and since $\epsilon_k \leq 2^{-(k+2)}$, we have that $\| \Phi - \Phi^k\|_{\infty} < 2 \epsilon_k$.  Fix $1 \leq i \leq n_k$.    The image of $\Omega_k^i$ under $\Phi$ is connected, thus
$$
\Phi( \Omega_k^i) \subseteq B_{2 \epsilon_k} ( \Phi^k (\Omega_k^i)),
$$
and hence, since~\eqref{29'} implies that $\Phi^k$ has Lipschitz constant at most $A_k + B_k + 2$,
$$
\lambda (\Phi (\Omega_k^i)) \leq \lambda(B_{2 \epsilon_k} ( \Phi^k(\Omega_k^i))) \leq 4 \epsilon_k + \lambda( \Phi^k (\Omega_k^i)) \leq 4 \epsilon_k + (A_k + B_k + 2) \mathrm{diam}(\Omega_k^i).
$$
So, summing over $1 \leq i \leq n_k$ gives, by the choice of $\epsilon_k$,~\ref{diamomega}, and~\eqref{measvk}, and since the $\{\Omega_k^i\}_{i=1}^{n_k}$ are pairwise disjoint,
\begin{align*}
\lambda (\Phi (\Omega_k))
& \leq \sum_{i=1}^{n_k}\lambda (\Phi (\Omega_k^i))\\
& \leq \sum_{i=1}^{n_k}\left(4 \epsilon_k + (A_k + B_k + 2) \mathrm{diam}(\Omega_k^i)\right)\\
& \leq 2^{-k} + (A_k + B_k + 2) (h_k + 2)\lambda(V_k)\\
& \leq 2^{-(k-1)}.
\end{align*}
Therefore, since for all $k \geq 0$,
$$
(\Phi \circ \tilde{U})(\tilde{U}^{-1}(U(E))) = \Phi (\tilde{U}(a,b)\cap U(E)) \subseteq \Phi (\Omega_k),
$$
we see that $(\Phi \circ \tilde{U})(\tilde{U}^{-1}(U(E))) $ is indeed a null set.
\section{Conclusion}
\begin{proof}[Proof of Theorem~\ref{sing:main}]
We let $L \in C^{\infty}(\mathbb{R}^3)$ be the Lagrangian given by Lemma~\ref{lemma11'}, with $u$ as constructed in Lemma~\ref{uuklemma}, $S = U(E)$, $\psi$ as given by Lemma~\ref{psilemma}, and this potential $\Phi$.  The growth condition on $\psi$ follows from~\ref{psi-2}.  By~\ref{psi-3},  we infer from Lemma~\ref{lemma11'} that the first statement of the theorem holds for this $u \in \mathrm{AC}(a,b)$.

Since $u' \in C^{\infty}([a,b] \setminus E)$ satisfies $u'(x) \to \infty$ as $\mathrm{dist} (x, E) \to 0$, we see that  $E = \{ x \in (a,b) : |u'(x)| = \infty \}$.  In particular, since our function $u$ is a minimizer with respect to its own boundary conditions, we see that the singular set of $u$ is indeed~$E$.

We now prove the third statement of the theorem.  Lemma~\ref{uuklemma} gives us a sequence of admissible functions $u_k \in C^{\infty}([a,b])$ which converge uniformly to $u$.  We just need to prove that they also converge in energy.  Let $ \epsilon > 0$.
By~\ref{u=ukoffvk} we see that
\begin{align*}
0  \leq \mathscr{L}(u_k) - \mathscr{L}(u) & = \int_a^b L(x, u_k(x), u_k'(x)) - L(x, u(x),u'(x))\, dx \\
& = \int_{V_k} L(x, u_k(x), u_k'(x)) - L(x, u(x), u'(x)) \, dx.
\end{align*}
We know from the precise conclusion of Lemma~\ref{lemma11'} that $x \mapsto L(x, u(x), u'(x))$ is integrable, so since $\lambda(V_k) \to 0$ as $ k \to \infty$ by~\eqref{measvk}, we can choose $k_0 \geq 1$ such that
$
\int_{V_{k}} L(x, u(x), u'(x))\, dx < \epsilon /2
$ whenever $k \geq k_0$.

Now, for each $k \geq 1$ and almost every $x \in [a,b]$, we have that
\begin{align*}
L(x, u_k(x), u_k'(x)) & = \omega(u_k'(x)) + F(x,u_k(x), u_k'(x)) \\
& = \omega(u_k'(x)) + \gamma(u_k'(x), \xi (x, u_k (x)), \theta (x, u_k (x)))
\end{align*}
by definition of the Lagrangian $L$ in Lemma~\ref{lemma11'}.  Fix such an $x \in [a,b]$.  We get the following upper bound for $\gamma$ by using~\ref{sing:4.4.e},~\ref{u'offvk},~\eqref{xxix} (noting $\xi \geq 0$ by~\eqref{xibignearS}), and~\ref{11beta}:
\begin{align*}
\gamma (u_k'(x), \xi(x, u_k (x)), \theta(x, u_k(x))) & \leq \theta(x, u_k (x)) | u_k'(x) - \xi(x, u_k (x)) + 1| \\
& \leq \Phi_y (x, u_k(x))(u_k'(x) + 1) + \theta (x, u_k(x))\xi(x, u_k (x)) \\
& \leq \Phi_y(x, u_k (x)) ((h_k + 2) +1) - \Phi_x (x, u_k (x))\\
& \leq -\Phi_x(x, u_k (x))(h_k + 7)/4 \\
& \leq -\Phi_x (x, u_k (x))h_k/2.
\end{align*}
Now,  $\omega(u_k'(x)) \leq \omega(h_k + 2)$ by properties of $\omega$ and~\ref{u'offvk}, and for sufficiently large $k \geq0$, $h_k +2 \leq \omega(h_k + 2)$, since $\omega$ is superlinear and $h_k \to \infty$ as $k \to \infty$.  So, again using~\ref{u'offvk}, and since certainly $-\Phi_x \geq 2$, we have, for large $k \geq 0$,
$$
L(x, u_k(x), u_k'(x))  \leq \omega(h_k + 2) -\Phi_x(x, u_k (x)) h_k/2   \leq -\Phi_x (x, u_k (x))\omega(h_k + 2).
$$
Now, if $ x \in \overline{V_{l-1}} \setminus \overline{V_{l}}$, then $(x, u_k(x)) \in \overline{\Omega_{l-1}} \setminus \overline{\Omega_{l}}$, so $ \Phi_x(x, u_k(x)) = \Phi_x^{l}(x, u_k(x))$, and hence $- \Phi_x (x, u_k(x)) \leq A_{l} + 1$.  Thus for large $l \geq 1$, almost everywhere on $\overline{V_{l-1}} \setminus \overline{V_{l}}$ we have
$$
L(x, u_k(x), u_k'(x)) \leq (A_{l} + 1)\omega(h_k + 2).
$$
So for sufficiently large $k\geq 1$, we have, since $\{h_k\}_{k=0}^{\infty}$ is increasing, by properties of $\omega$, and~\eqref{measvk}, that
\begin{align*}
0 \leq \int_{V_k} L(x, u_k(x), u_k'(x))\, dx & \leq \sum_{l=k + 1}^{\infty} \int_{\overline{V_{l-1}} \setminus \overline{V_{l}}} L(x, u_k(x), u_k'(x))\, dx\\
& \leq \sum_{l=k+1}^{\infty} \int_{\overline{V_{l-1}} \setminus \overline{V_{l}}} (A_{l} + 1) \omega(h_k + 2)\, dx\\
& \leq  \sum_{l=k+1}^{\infty} \omega(h_l + 2) (A_{l} + 1) \lambda(V_{l-1}) \\
& \leq \sum_{l=k}^{\infty} 2^{-l}\\
& \leq 2^{-k +1}.
\end{align*}
So choosing $k_1 \geq 1$ such that $2^{-k_1 + 1} \leq \epsilon / 2$, we have for large $k\geq k_0, k_1$ that
$$
0 \leq \mathscr{L}(u_k) - \mathscr{L}(u)  \leq \int_{V_k} L(x, u_k(x), u_k'(x))\, dx + \int_{V_k}L(x, u(x), u'(x))\, dx  \leq  \epsilon,
$$
as required.
\end{proof}
\def\cprime{$'$}

\end{document}